\documentclass[12pt]{article} 

\usepackage{latexsym,amssymb,amsthm,amsmath}
\usepackage{url,graphics}
\usepackage[dvips]{graphicx,epsfig}

\usepackage[T2A]{fontenc}
\usepackage[utf8]{inputenc}
\usepackage{amsfonts}
\usepackage{amsthm}
\usepackage{amsmath}
\usepackage{amssymb}
\usepackage[unicode=true]{hyperref}
\hypersetup{
colorlinks=true,
breaklinks=true,
linkcolor={black},
citecolor={black},
urlcolor={blue!60!black},
pdfauthor={Nikolai Karol}} 
\usepackage{pdfpages}
\usepackage{dsfont}
\usepackage{scalerel}
\usepackage{makecell,xcolor}
\usepackage[12pt]{extsizes} 
\usepackage{graphicx} 
\graphicspath{{pictures/}} 
\DeclareGraphicsExtensions{.pdf} 
\DeclareGraphicsExtensions{.png} 
\usepackage[utf8]{inputenc} 
\usepackage[russian, german, english]{babel}
\usepackage{amsmath}
\usepackage{indentfirst}
\usepackage{amsthm, amsmath, amssymb} 
\usepackage{tikz} 
\usepackage{amsfonts}
\usepackage[left=2cm,right=2cm,top=2cm,bottom=2cm,bindingoffset=0cm]{geometry} 
\usepackage{wrapfig} 
\usepackage{graphicx}
\graphicspath{{pictures/}} 
\DeclareGraphicsExtensions{.pdf} 
\DeclareGraphicsExtensions{.png} 
\usepackage{caption} 

\usepackage[longnamesfirst,numbers,sort&compress]
{natbib}
\graphicspath{[pictures/]} 
\DeclareGraphicsExtensions{.pdf,.png,.jpg}

\title{\textbf{Restriction on minimum degree in the contractible sets problem}}

\author{Nikolai Karol \thanks{The work is supported by Ministry of Science and Higher Education of the Russian Federation, agreement № 075–15–2022–287.}}

\date{}

\begin{document}

\maketitle

\righthyphenmin=2
\newtheorem{thm}{Theorem}
\newtheorem{lem}{Lemma}
\renewcommand*{\proofname}{\bf Proof}
\newtheorem{cor}{Corollary}
\newtheorem{conj}{Conjecture}
\newtheorem{claim}{Claim}
\newtheorem{subcl}{Subclaim}[claim]
\theoremstyle{definition}
\newtheorem{defin}{Definition}
\theoremstyle{remark}
\newtheorem{rem}{\bf Remark}

\def\mmax{\mathop{\rm max}}
\def\q#1.{{\bf #1.}}
\def\P{{\rm Part}}
\def\QP{{\rm QPart}}
\def\I{{\rm Int}}
\def\R{{\rm Bound}}
\def\RC{{\rm Cut}}
\def\B{{\rm BT}}
\def\T{{\rm T}}
\def\N{{\rm N}}
\def\GM{{\cal GM}}

\begin{center}
    \textbf{Abstract}
\end{center}

\small{Let $G$ be a $3$-connected graph. A set $W \subset V(G)$ is called \textit{contractible} if $G(W)$ is a connected graph and $G - W$ is a $2$-connected graph. In 1994, McCuaig and Ota conjectured that for any $k \in \mathbb{N}$ there exists $n \in \mathbb{N}$ such that any 3-connected graph $G$ with $v(G) \geqslant n$ has a $k$-vertex contractible set. It is proved that this holds if $k \geqslant 5$ and $\delta(G) \geqslant \left[ \frac{2k + 1}{3} \right] + 2$.}



\section{Introduction}

\normalsize{We consider undirected graphs without loops and multiple edges and use standard notation. We denote by $v(G)$ the number of vertices of $G$ and by $\delta(G)$ the minimum degree of $G$.}

\begin{defin} 
Let $R \subseteq V(G)$.

1) $G - R$ is the graph obtained from $G$ by deleting all vertices of $R$ and all edges incident to the vertices of $R$.

2) $G(R)$ is the induced subgraph of $G$ on $R$.

3) $R$ is said to be \textit{connected} if $G(R)$ is connected.

4) $R$ is said to be a $k$\textit{-vertex set} if $|R| = k$.

5) $R$ is said to be \textit{contractible} if $G(R)$ is connected and $G - R$ is 2-connected.

6) $R$ is said to be $k$\textit{-contractible} if $R$ is a $k$-vertex contractible set.

7) Let $R_{1} \subseteq V(G)$ be a set such that $R \cap R_{1} = \varnothing$. We denote by $E_{G}(R, R_{1})$ the set of all $e \in E(G)$ that $e = xy$ with $x \in R$ and $y \in R_{1}$. Let $e_{G}(R, R_{1}) = |E_{G}(R, R_{1})|$. We say that $R_{1}$ is \textit{adjacent} to $R$ if $e_{G}(R, R_{1}) \geqslant 1$.

\end{defin}

Consider a $2$-connected graph $G$ with $n$ vertices, and let $n_{1}$ and $n_{2}$ be positive integers such that $n_{1} + n_{2} = n$. It is a well-known fact that $V(G)$ can be partitioned into two connected sets $V_{1}$ and $V_{2}$ such that $|V_{1}| = n_{1}$ and $|V_{2}| = n_{2}$.

In 1994, McCuaig and Ota~\cite{McOta1} formulated the following conjecture for 3-connected graphs. This conjecture was mentioned in Mader’s survey on connectivity~\cite{Mader2}.

\begin{conj} [\citep{McOta1}] \label{conjectureconjecture} \textit{For each $k \in \mathbb{N}$, there exists an integer $n$ such that every $3$-connected graph G on at least $n$ vertices has a $k$-contractible set.}
\end{conj}

Mader~\cite{Mader3} showed that the answer to the analogous problem is negative for $t$-connected graphs with $t \geqslant 4$. More specifically, for any $k \geqslant 2$, there exists an arbitrarily large $t$-connected graph $G$ such that $G$ does not contain a connected set $W$ such that $|W| = k$ and $G - W$ is $(t - 1)$-connected. Thus the question remains open only for $3$-connected graphs.

Conjecture \ref{conjectureconjecture} is clear for $k = 1$, and is proved for $k = 2$ in~\cite{Tutte4}, for $k = 3$ in~\cite{McOta1}, for $k = 4$ in~\cite{Kriesell5} and for $k = 5$ in~\cite{Nadya6}. 

Karpov~\cite{Karpov7} established the existence of large contractible sets in $3$-connected graphs.

\begin{thm} [\citep{Karpov7}] Let $n \geqslant 5$ be an integer, and let $G$ be a 3-connected graph with $v(G) \geqslant 2n + 1$. Then $G$ has a contractible set $W$ such that $n \leqslant |W| \leqslant 2n - 4$.
\end{thm}

Kriesell~\cite{Kriesell5} proved the following.

\begin{thm} [\citep{Kriesell5}] \label{kriessell'sresult} Let $G$ be a 3-connected graph with at least 7 vertices that is not isomorphic to~$K_{3, 4}$. Then $G$ has a 4-contractible set.
\end{thm}

Our main result is the following.

\begin{thm} \label{theoremtheorem} For any integer $k \geqslant 5$, every 3-connected graph $G$ with $v(G) \geqslant k + 3$ and $\delta(G) \geqslant  \left[ \frac{2k + 1}{3} \right] + 2$ has a $k$-contractible set.
\end{thm}

\section{Tools}

We formulate several definitions and facts concerning the structure of $n$-connected graphs. 

\begin{defin}

A contractible set $W \subseteq V(G)$ of a $3$-connected graph $G$ is \textit{maximal} if there exists no vertex $x \in V(G) \setminus W$ such that $W \cup \{x\}$ is contractible.

\end{defin}

\begin{defin}

Let $G$ be an $n$-connected graph.

1) A set $R \subseteq V(G)$ is a \textit{cutset} if $G - R$ is not connected.

2) We denote by $\mathfrak{R}_{n}(G)$ the set of all $n$-vertex cutsets of $G$.

3) A cutset $R \subseteq V(G)$ \textit{splits} a set $X \subseteq V(G)$ if $X \setminus R$ is not contained in one connected component of $G - R$.

4) Two cutsets $S, T \in \mathfrak{R}_{n}(G)$ are \textit{independent} if $S$ does not split $T$ and $T$ does not split $S$. Otherwise, these cutsets are \textit{dependent}.

\end{defin}

\begin{defin} Let $\mathfrak{S} \subset \mathfrak{R}_{n}(G)$.

1) A set $A \subseteq V(G)$ is a \textit{part of decomposition} of $G$ by $\mathfrak{S}$ if no cutset of $\mathfrak{S}$ splits $A$ and $A$ is a maximal set up to inclusion with this property. By Part$(G; \mathfrak{S})$, we denote the set of all parts of decomposition of $G$ by $\mathfrak{S}$.

2) Let $A \in$ Part$(G; \mathfrak{S})$. A vertex of $A$ is \textit{inner} if it does not belong to any cutset of $\mathfrak{S}$. The set of all inner vertices of $A$ is called the \textit{interior} of $A$ and is denoted by Int($A$).

The \textit{boundary} of $A$ is the set Bound($A$) $= A \setminus \text{Int}(A)$.

\end{defin}

\begin{defin} Let $G$ be a $2$-connected graph.

1) A cutset $S \in \mathfrak{R}_{2}(G)$ is \textit{single} if $S$ is independent with all other cutsets of $\mathfrak{R}_{2}(G)$. We denote by $\mathfrak{O}(G)$  the set of all single cutsets of $G$.

2) We write Part($G$) instead of Part($G; \mathfrak{O}(G)$). The parts of this decomposition are simply called the \textit{parts} of $G$.

\end{defin}

\begin{defin} The \textit{block tree} BT($G$) of a $2$-connected graph $G$ is a bipartite graph with bipartition ($\mathfrak{O}(G)$, Part($G$)), where a single cutset $S$ and a part $A$ are adjacent if and only if $S \subset A$.

\end{defin}

We need the following property of BT($G$).

\begin{lem} {\rm \cite[Lemma 1]{Karpov8}} Let $G$ be a $2$-connected graph. Then BT($G$) is a tree and every leaf of BT($G$) corresponds to a part of Part($G$).
\end{lem}

\begin{defin}

Let $G$ be a $2$-connected graph. A part $A \in$ Part($G$) is \textit{pendant} if $A$ corresponds to a leaf of BT($G$).

\end{defin}

\begin{defin} Let $G$ be a $2$-connected graph.

1) We denote by $G'$ the graph obtained from $G$ by adding all edges $ab$ where $\{a, b\} \in \mathfrak{O}(G)$.

2) A part $A \in$ Part($G$) is called a \textit{cycle} if the graph $G'(A)$ is a cycle. If $A$ is a cycle then $|A|$ is the \textit{length} of $A$.

\end{defin}

\begin{lem} \label{lemma2} {\rm \cite[Lemma 13]{Karpov7}} Let $G$ be a $3$-connected graph. Let $W \subset V(G)$ be a maximal contractible set such that the graph $H = G - W$ is not a simple cycle. Then the following statements hold.

1) Let $A \in$ $Part(H)$ be a cycle. Then each inner vertex of $A$ is adjacent to $W$.

2) There are at least two pendant parts in $Part(H)$ and all these parts are cycles of length at least 4.

3) Let $A \in Part(H)$ be a pendant part. Then $H - Int(A)$ is $2$-connected.
\end{lem}

The following lemma is a direct corollary of Lemma~\ref{lemma2}. The original version of Lemma~\ref{lemma3} was proved by Kriesell~\citep[Lemma 3]{Kriesell5}.

\begin{lem} \label{lemma3} Let $G$ be a $3$-connected graph. Let $W \subset V(G)$ be a maximal contractible set such that $G - W$ is not a simple cycle. Let $A_{1}, A_{2}$ be two pendant parts of $G - W$, and let $W_{1} = Int(A_{1})$ and $W_{2} = Int(A_{2})$. Then the following statements hold.

1) $G(W_{1})$ and $G(W_{2})$ are simple paths.

2) $|W_{1}| \geqslant 2, |W_{2}| \geqslant 2$.

3) $W_{1} \cap W_{2} = \varnothing$.

4) All vertices in $W_{1} \cup W_{2}$ have degree 2 in $G - W$.

5) Both $G - W - W_{1}$ and $G - W - W_{2}$ are 2-connected.

6) $N_{G}(W_{1}) \cap W_{2} = \varnothing, N_{G}(W_{2}) \cap W_{1} = \varnothing$.
\end{lem}

\section{Proof of Theorem~\ref{theoremtheorem}}

\begin{lem} \label{firstlemma} Let $G$ be a $3$-connected graph and $k \geqslant 2$ be an integer such that $v(G) \geqslant k + 3$ and $G$ has a $(k - 1)$-contractible set $W$. Assume that there exist four distinct vertices $v_{1}, v_{2}, v_{3}, v_{4} \in V(G - W)$ such that:

1) $v_{1}v_{2} \in E(G)$ and $d_{G - W}(v_{i}) = 2$ for each $i \in \{1, 2, 3, 4\}$,

2) for any $x \in W$ such that $xv_{3}, xv_{4} \in E(G)$, the graph $G - (W \setminus \{x\}) - \{v_{1}, v_{2}\}$ is 2-connected,

3) $|N_{G}(v_{3}) \cap N_{G}(v_{4}) \cap W| > |W \setminus N_{G}(\{v_{1}, v_{2}\})|$.

Then $G$ has a $k$-contractible set.

\end{lem}

\begin{proof} We have $d_{G - W}(v_{i}) = 2$ for any $i \in \{1, 2, 3, 4\}$ and $\delta(G) \geqslant 3$. Therefore, $e_{G}(v_{i}, W) \geqslant 1$ for each $i \in \{1, 2, 3, 4\}$. By condition 3) of Lemma~\ref{firstlemma}, $|N_{G}(v_{3}) \cap N_{G}(v_{4}) \cap W| > 0$. Let $x \in W$ be a common neighbour of $v_{3}$ and $v_{4}$. Note that there are $|N_{G}(v_{3}) \cap N_{G}(v_{4}) \cap W|$ candidates on $x$.

Consider $\{v_{1}, v_{2}\} \cup (W \setminus \{x\})$ (see Figure~\ref{figurezero}). Observe that $|\{v_{1}, v_{2}\} \cup (W \setminus \{x\})| = k$. By condition of Lemma~\ref{firstlemma}, $G - (\{v_{1}, v_{2}\} \cup (W \setminus \{x\}))$ is 2-connected. Therefore the set $\{v_{1}, v_{2}\} \cup (W \setminus \{x\})$ is not contractible only if it is not connected. By assumption, $v_{1}v_{2} \in E(G)$, and hence there is a connected component of the graph $G(W \setminus \{x\})$ such that all its vertices are not adjacent to~$\{v_{1}, v_{2}\}$.

\begin{figure}[ht]
	\centering
		\includegraphics[width=0.31\columnwidth, keepaspectratio]{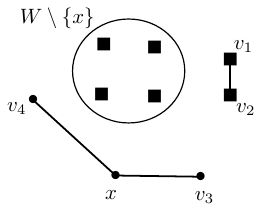}
     \caption{Proof of Lemma~\ref{firstlemma}.}
	\label{figurezero}
\end{figure}

Consequently, we need to take $x$ such that there is no connected component of non-adjacent to $\{v_{1}, v_{2}\}$ vertices in $G(W \setminus \{x\})$ (we call such vertices in $G(W)$ \textit{forbidden}). Observe that there are $|W \setminus N_{G}(\{v_{1}, v_{2}\})|$ forbidden vertices. Recall that there are $|N_{G}(v_{3}) \cap N_{G}(v_{4}) \cap W|$ candidates on~$x$. Then by condition 3) of Lemma~\ref{firstlemma}, there are more candidates on $x$ than forbidden vertices.

For each forbidden vertex, we take a shortest path from this forbidden vertex to the set of vertices in $W$ which are not forbidden (choosing arbitrarily if several such paths exist). For every such path, we take a neighbour of the forbidden endpoint (the second vertex of the path); let $P$ be the union of these vertices. Observe that $|P|$ is at most the number of forbidden vertices. Recall that there are more candidates on $x$ than forbidden vertices. Consequently, there exists $x \notin P$ such that $xv_{3} \in E(G)$ and $xv_{4} \in E(G)$; we fix this $x$. Note that $x$ is suitable. Indeed, assume for the sake of contradiction that there exists a connected component in $G(W \setminus \{x\})$ consisting only of forbidden vertices. Since $G(W)$ is connected, $x$ has at least one neighbour in this component. For this neighbour, $x$ must be the second vertex of the path mentioned above, and hence $x \in P$, a contradiction.
\end{proof}

We prove Theorem \ref{theoremtheorem} by induction on $k$. For convenience, we formulate the induction step as a separate statement.

\begin{lem} \label{mainlemma} Let $k$ be an integer and let $c$ be a non-negative integer such that $k \geqslant 3c + 5$. Let $G$ be a $3$-connected graph such that $G$ has a $(k - 1)$-contractible set, $v(G) \geqslant k + 3$, and $\delta(G) \geqslant k - c$. Then $G$ has a $k$-contractible set.
\end{lem}

\begin{proof}

Let $W$ be a $(k - 1)$-contractible set in $G$. Assume for the sake of contradiction that $W$ is maximal.

\vspace{3.5mm}

\noindent \textbf{Case 1.} $G - W$ is a simple cycle.
\vspace{3.5mm}

We enumerate the vertices of $G - W$ in a cyclic order: $r_{1}, r_{2}, ...$, $r_{m}$, where $m \geqslant 4$. Our purpose is to apply Lemma~\ref{firstlemma}, setting $v_{1} = r_{i + 1}$, $v_{2} = r_{i + 2}$, $v_{3} = r_{i}$, $v_{4} = r_{i + 3}$ for arbitrary $i$.

Observe that condition 1) of Lemma~\ref{firstlemma} is satisfied.

We now verify condition 2) of Lemma~\ref{firstlemma}. Assume that $r_{i}$ and $r_{i + 3}$ have a common neighbour $x$ in $W$ (see Figure~\ref{figureone}). Then $G - \{r_{i + 1}, r_{i + 2}\} - (W \setminus \{x\})$ is 2-connected because this graph has a Hamiltonian cycle $r_{i}xr_{i + 3}r_{i + 4}....r_{i - 1}$.

\begin{figure}[ht]
	\centering
		\includegraphics[width=0.35\columnwidth, keepaspectratio]{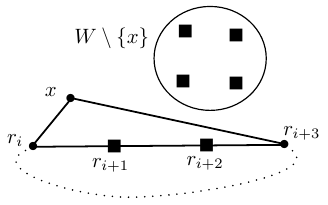}
     \caption{$G - W$ is a simple cycle.}
	\label{figureone}
\end{figure}

It remains to verify condition 3) of Lemma~\ref{firstlemma}. Since $G - W$ is a simple cycle and $\delta(G) \geqslant k - c$, we have $e_{G}(r_{i}, W) \geqslant k - c - 2$, $e_{G}(r_{i + 1}, W) \geqslant k - c - 2$, and $e_{G}(r_{i + 3}, W) \geqslant k - c - 2$. Since $|W| = k - 1$ and $2(k - c - 2) - (k - 1) > c + 1$, we have $|N_{G}(r_{i}) \cap N_{G}(r_{i + 3}) \cap W| \geqslant c + 2$. Taking into account $e_{G}(r_{i + 1}, W) \geqslant k - c - 2$, this yields $|W \setminus N_{G}(\{r_{i + 1}, r_{i + 2}\})| \leqslant c + 1$. Consequently, $|N_{G}(r_{i}) \cap N_{G}(r_{i + 3}) \cap W| > |W \setminus N_{G}(\{r_{i + 1}, r_{i + 2}\})|$. Hence, condition 3) of Lemma~\ref{firstlemma} is satisfied.

Thus Lemma~\ref{firstlemma} can be applied, and $G$ has a $k$-contractible set.

\vspace{3.5mm}

\noindent \textbf{Case 2.} $G - W$ is not a simple cycle.

\vspace{3.5mm}

By Lemma~\ref{lemma2}, $G - W$ has at least two pendant parts, and each of them is a cycle of length at least 4. Therefore both parts have at least two internal vertices. 

\vspace{3.5mm}

\noindent \textbf{Case 2.1.} There exists a pendant part $A$ of $G - W$ such that $|A| \geqslant 5$ (in particular, $|Int(A)| \geqslant 3$).

\vspace{3.5mm}

Let $Bound(A) = \{r, s\}$ and $Int(A) = \{w_{1}, ..., w_{l}\}$, where the vertices are enumerated so that $rw_{1}...w_{l}s$ is a path. Let $u$ be an internal vertex of the pendant part different from $A$. By Lemma~\ref{lemma3}.4, $d_{G - W}(w_{i}) = 2$ for each $i \in \{1, 2, 3\}$, and $d_{G - W}(u) = 2$.

Our goal is to apply Lemma~\ref{firstlemma}, where $v_{1} = w_{1}$, $v_{2} = w_{2}$, $v_{3} = w_{3}$, $v_{4} = u$. Observe that condition 1) of Lemma~\ref{firstlemma} holds.

We now verify condition 2) of Lemma~\ref{firstlemma}. Assume that $w_{3}$ and $u$ have a common neighbour $x \in W$ (see Figure~\ref{figuretwo}). Then $G - (\{w_{1}, w_{2}\} \cup (W \setminus \{x\}))$ is 2-connected because $G - W - Int(A)$ is 2-connected (by Lemma~\ref{lemma2}.3) and there is a path $sw_{l}....w_{3}xu$.

\begin{figure}[ht]
	\centering
		\includegraphics[width=0.4\columnwidth, keepaspectratio]{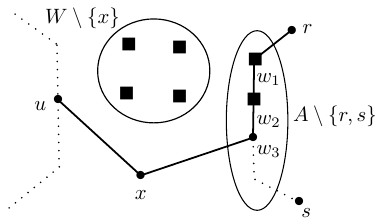}
     \caption{$G - W$ is not a simple cycle, and there exists a pendant part $A$ of $G - W$ such that $|A| \geqslant 5$.}
	\label{figuretwo}
\end{figure}

It remains to verify condition 3) of Lemma~\ref{firstlemma}. Recall that $\delta(G) \geqslant k - c$, $d_{G - W}(w_{i}) = 2$ for any $i \in \{1, 2, 3\}$, and $d_{G - W}(u) = 2$. Therefore, $e_{G}(w_{i}, W) \geqslant k - c - 2$ for each $i \in \{1, 2, 3\}$, and $e_{G}(u, W) \geqslant k - c - 2$. Then it follows from $|W| = k - 1$ and $2(k - c - 2) - (k - 1) > c + 1$ that $|N_{G}(w_{3}) \cap N_{G}(u) \cap W| \geqslant c + 2$. Since $e_{G}(w_{1}, W) \geqslant k - c - 2$, we have $|W \setminus N_{G}(\{w_{1}, w_{2}\})| \leqslant c + 1$. Consequently, $|N_{G}(w_{3}) \cap N_{G}(u) \cap W| > |W \setminus N_{G}(\{w_{1}, w_{2}\})|$. Hence, condition 3) of Lemma~\ref{firstlemma} holds.

Therefore Lemma~\ref{firstlemma} can be applied, and $G$ has a $k$-contractible set.

\vspace{3.5mm}

\noindent \textbf{Case 2.2.} Each pendant part of $G - W$ consists of four vertices.

\vspace{3.5mm}

\begin{rem} In this case, we need $k \geqslant 2c + 4$ only, instead of $k \geqslant 3c + 5$.
\end{rem}

Let $N_{v} = N_{G}(v) \cap W$ for any vertex $v \in G - W$.

Let $A$ and $B$ be two pendant parts of $G - W$. Let $\{u_{1}, u_{2}\} = Int(A)$ and $\{w_{1}, w_{2}\} = Int(B)$. By Lemma~\ref{lemma3}.4, $d_{G - W}(u_{1}) = 2$, $d_{G - W}(u_{2}) = 2$, $d_{G - W}(w_{1}) = 2$, and $d_{G - W}(w_{2}) = 2$. Since $\delta(G) \geqslant k - c$, we have $e_{G}(u_{1}, W) \geqslant k - c - 2$, $e_{G}(u_{2}, W) \geqslant k - c - 2$, $e_{G}(w_{1}, W) \geqslant k - c - 2$, and $e_{G}(w_{2}, W) \geqslant k - c - 2$.

\vspace{3.5mm}

\noindent \textbf{Case 2.2.1.} $|N_{w_{1}} \cap N_{w_{2}}| > |W \setminus (N_{u_{1}} \cup N_{u_{2}})|$ or $|N_{u_{1}} \cap N_{u_{2}}| > |W \setminus (N_{w_{1}} \cup N_{w_{2}})|$.

\vspace{3.5mm}

Without loss of generality, $|N_{w_{1}} \cap N_{w_{2}}| > |W \setminus (N_{u_{1}} \cup N_{u_{2}})|$. Our purpose is to apply Lemma~\ref{firstlemma}, setting $v_{1} = u_{1}$, $v_{2} = u_{2}$, $v_{3} = w_{1}$, and $v_{4} = w_{2}$. Observe that condition 3) of Lemma~\ref{firstlemma} is satisfied. Recall that $d_{G - W}(u_{1}) = 2$, $d_{G - W}(u_{2}) = 2$, $d_{G - W}(w_{1}) = 2$, and $d_{G - W}(w_{2}) = 2$. Hence, condition 1) of Lemma~\ref{firstlemma} holds.

It remains to verify condition 2) of Lemma~\ref{firstlemma}. Note that if $w_{1}$ and $w_{2}$ have a common neighbour $x$ (see Figure~\ref{figurethree}) then $G - \{u_{1}, u_{2}\} - (W \setminus \{x\})$ is 2-connected. Indeed, by Lemma~\ref{lemma2}.3, $G - \{u_{1}, u_{2}\} - W$ is 2-connected and $xw_{1}, xw_{2} \in E(G)$.

\begin{figure}[ht]
	\centering
		\includegraphics[width=0.33\columnwidth, keepaspectratio]{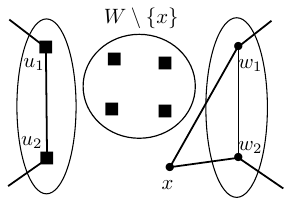}
     \caption{$G - W$ is not a simple cycle, and each pendant part of $G - W$ consists of four vertices.}
	\label{figurethree}
\end{figure}

Consequently, Lemma~\ref{firstlemma} can be applied, and $G$ has a $k$-contractible set.

\vspace{3.5mm}

\noindent \textbf{Case 2.2.2.} $|N_{w_{1}} \cap N_{w_{2}}| \leqslant |W \setminus (N_{u_{1}} \cup N_{u_{2}})|$ and $|N_{u_{1}} \cap N_{u_{2}}| \leqslant |W \setminus (N_{w_{1}} \cup N_{w_{2}})|$.

\vspace{3.5mm}

Denote $$f_{1} = |N_{u_{1}} \cap N_{u_{2}}|, f_{2} = |N_{u_{1}} \setminus N_{u_{2}}|, f_{3} = |N_{u_{2}} \setminus N_{u_{1}}|,$$ $$e_{1} = |N_{w_{1}} \cap N_{w_{2}}|, e_{2} = |N_{w_{1}} \setminus N_{w_{2}}|, e_{3} = |N_{w_{2}} \setminus N_{w_{1}}|.$$ 

It follows from the observation before case 2.2.1 that $|N_{v}| \geqslant k - c - 2$ for each $v \in \{u_{1}, u_{2}, w_{1}, w_{2}\}$. Therefore, 

$$f_{1} + f_{2} \geqslant k - c - 2, f_{1} + f_{3} \geqslant k - c - 2, e_{1} + e_{2} \geqslant k - c - 2, e_{1} + e_{3} \geqslant k - c - 2.$$

Combining all these inequalities, we obtain

$$2f_{1} + f_{2} + f_{3} + 2e_{1} + e_{2} + e_{3} \geqslant 4(k - c - 2).$$

By the condition of case 2.2.2, $|N_{w_{1}} \cap N_{w_{2}}| \leqslant |W \setminus (N_{u_{1}} \cup N_{u_{2}})|$. Therefore, $e_{1} \leqslant |W \setminus (N_{u_{1}} \cup N_{u_{2}})|$. By definition of $N_{v}$, we have $N_{u_{1}} \cup N_{u_{2}} \subseteq W$ for any $v \in V(G - W)$. Since $|W| = k - 1$, we have $e_{1} \leqslant k - 1 - |N_{u_{1}} \cup N_{u_{2}}|$. Hence, $e_{1} \leqslant k - 1 - f_{1} - f_{2} - f_{3}$. Therefore, we derive from $|N_{w_{1}} \cap N_{w_{2}}| \leqslant |W \setminus (N_{u_{1}} \cup N_{u_{2}})|$ that $e_{1} + f_{1} + f_{2} + f_{3} \leqslant k - 1$. Similarly, it follows from $|N_{u_{1}} \cap N_{u_{2}}| \leqslant |W \setminus (N_{w_{1}} \cup N_{w_{2}})|$ that $f_{1} + e_{1} + e_{2} + e_{3} \leqslant k - 1$. Combining these inequalities, we obtain

$$2f_{1} + 2e_{1} + e_{2} + e_{3} + f_{2} + f_{3} \leqslant 2k - 2.$$

Hence, $2k - 2 \geqslant 4(k - c - 2) \Rightarrow 2c + 3 \geqslant k$, a contradiction to $k \geqslant 2c + 4$. \end{proof}

\vspace{3.5mm}

\begin{proof}[\textbf{Proof of Theorem~\ref{theoremtheorem}:}]

\vspace{3.5mm}

We prove Theorem~\ref{theoremtheorem} by induction. The case $k = 4$ serves as the base case of the induction. We do not require any restriction on $\delta(G)$ in the base case. By Theorem~\ref{kriessell'sresult}, any $3$-connected graph on at least 8 vertices contains a $4$-contractible set. Hence, there is no problem with the restriction on the number of vertices in the induction step from $k = 4$ to $k = 5$.

\vspace{3.5mm}

\textbf{Induction step.}

\vspace{3.5mm}

By the induction hypothesis, $G$ has a $(k - 1)$-contractible set. Let $c = \left[ \frac{k - 5}{3} \right]$. Then $k \geqslant 3c + 5$ and $\delta(G) \geqslant k - c$. By Lemma~\ref{mainlemma}, $G$ has a $k$-contractible set.
\end{proof}

\end{document}